\documentclass[11pt]{article} 

\usepackage[utf8]{inputenc} 
  
\usepackage{geometry} 
\usepackage{booktabs} 
\usepackage{array} 
\usepackage{paralist} 
\usepackage{verbatim} 
\usepackage{amsmath}
\usepackage{amssymb}
\usepackage{amsthm}
\usepackage{xifthen}
\usepackage{xcolor}
\usepackage{graphicx} 
\usepackage{epstopdf}
\usepackage[font=scriptsize]{caption}
\usepackage{soul}
\usepackage[normalem]{ulem}

\usepackage{subfigure}

\usepackage[colorlinks=true,linkcolor=blue,urlcolor=blue]{hyperref}

\usepackage{fancyhdr} 
\usepackage{sectsty}
\allsectionsfont{\sffamily\mdseries\upshape} 

\newtheorem{theorem}{Theorem}

\newtheorem{definition}[theorem]{Definition}

\newtheorem{lemma}[theorem]{Lemma}

\newtheorem{remark}[theorem]{Remark}

\newtheorem*{acknowledgement}{Acknowledgement}

\usepackage[nottoc,notlof,notlot]{tocbibind} 
\usepackage[titles,subfigure]{tocloft} 

\newcommand{\E}{\mathbb E}

\newcommand{\F}{\mathcal F}

\renewcommand{\P}{\mathbb P}

\newcommand{\kom}[1]{}
\renewcommand{\kom}[1]{{\bf [#1]}}

\newcounter{komcounter}
\numberwithin{komcounter}{section}

\title{
Hiring and firing -- a signaling game}
\author{Erik Ekström\\\textit{Department of Mathematics, Uppsala University}\\ \\
Topias Tolonen-Weckström\\\textit{Department of Mathematics, Uppsala University}
}
\date{February 12, 2024} 

\begin{document}

\maketitle

\begin{abstract}
We study a signaling game between an employer and a potential employee, where the employee has private information regarding their production capacity. At the initial stage, the employee communicates a salary claim, after which 
the true production capacity is gradually revealed to the employer as the unknown 
drift of a Brownian motion representing the revenues generated by the employee.
Subsequently, the employer has the possibility to choose a time to fire the employee 
in case the estimated production capacity falls short of the salary. In this set-up, we use filtering and optimal stopping theory to
derive an equilibrium in which the employee provides a randomized salary claim and the employer uses a threshold strategy in terms of the conditional probability for the high production capacity. 
The analysis is robust in the sense that various extensions of the basic model can be solved using the same methodology, including 
cases with positive firing costs, incomplete information about one's own type, as well as an additional interview phase.
\end{abstract}

\section{Introduction}\label{intro}

Incomplete information is a key ingredient in many hiring processes, where full knowledge about the true capacity of a potential employee is rarely available to the employer at the hiring time. Instead, if the candidate is hired, such information
is gradually revealed to the employer over time. On the other hand, the potential employee would typically 
possess more accurate information, and would use this additional information when providing their salary claim.
Naturally, a high salary is costly for the employer, and thus increases the risk for the employee of being fired.
Therefore, there
is a trade-off in the choice between a high salary claim to increase personal income and a small salary claim to decrease the risk of being fired.

To model one possible instance of the strategic interaction between an employer and a potential employee, 
we set up and study a stochastic game with asymmetric information between two players.
The game is informally described as follows. 
\begin{itemize}
\item[(i)]
The capacity $\mu$ of the employee (Player 1) is a random variable with a known two-point distribution.
\item[(ii)]
At time $t=0$, Player 1 learns about the realization of the random variable $\mu$, and presents to the employer (Player 2) a non-negotiable salary claim $C$; the salary can only take two values.
\item[(iii)]
At time 0, Player 2 observes the salary claim, and subsequently also noisy observations of $\mu$, based upon which 
a choice is made of a stopping time $\tau$ to terminate the employment; here $\tau=0$ corresponds to a case in which the salary claim is not accepted (no hiring), $0<\tau<\infty$ corresponds to an accepted salary claim, but with
firing in finite time, and $\tau=\infty$ to an accepted salary claim, with no firing taking place. 
\item[(iv)]
Up to the termination time $\tau$, Player~1 receives compensation at the chosen rate $C$ per unit of time. Player~2, on the other hand, earns a net payment stream consisting of increments of a stochastic process $\mu t+ \sigma W_t-Ct$, where $W$ is a Brownian motion which accounts for random fluctuations about the mean rate $\mu-C$.
\end{itemize}

The above game is a {\em signaling game}, with two possible types of Player 1, corresponding to the two possible values of $\mu$, and where Player 1 sends a signal by choosing the salary level $C$. As such, there is an incomplete and asymmetric information structure since the players have different knowledge about $\mu$. 

Variants of such signaling games with asymmetric information have a long history in literature on hiring of staff and salary formation.
A classical reference is \cite{S},
where an example with a job seeker that can have two different types is studied. The job seeker knows their own type and
chooses an education level, where the cost of education depends on the type, thereby conveying information to the employer. 
In the set-up of \cite{S}, all information is conveyed at the initial time,
and the signal that consists of the chosen education level does not influence the actual ability of the employee.
\cite{AP} and \cite{DG2} study extensions with a type-dependent continuation value for the job-seeker, thus 
allowing for a future change in the salary level.
In particular, the set-up includes cases of 
employer learning based on an additional report (such as an interview phase) or based on the employee performance on-the-job. For related studies allowing for uncertainty about one's own type, see \cite{W}, and for more possible types of the job-seeker and competition between employers, see \cite{DGKM}.
Also, a signaling game outside of the job market is explored in \cite{DG}, where an owner of a company and several
potential buyers are considered. The seller holds private information of the company type, and buyers learn gradually 
from noisy observations of the unknown type and from the actions (lack of actions) of the seller. 
The hiring problem that we consider is also related to so-called principal-agent problems (see e.g. \cite{CPT}, \cite{HM} and \cite{Sa}), in which a principal seeks to set-up a compensation scheme to control the effort of the agent. 

The methods we use to derive a PBE of the signaling game rely on a combination of stochastic filtering and stochastic control theory. For stochastic filtering of the drift of a diffusion process we refer to \cite{LS}, and for a classical application to sequential hypothesis testing of the drift of Brownian motion, where filtering and optimal stopping theory is combined, see \cite[Chapter 4]{Sh}. For other studies of combined filtering and control, see e.g.  \cite{DA}, \cite{EV} and \cite{La} for single-agent problems, and \cite{CR}, \cite{ELO} and \cite{G} for strategic set-ups.

In contrast to the set-up of \cite{S} and many subsequent papers, the current study does not use education level as a signal. 
Instead, we equip the employee with the right to provide a salary claim, which then is paid out continuously until a possible firing time. 
This allows us to study the trade-off 
between a high personal income and the risk of becoming a burden for the company.
In line with the literature on signaling games as above, see also \cite{FT}, \cite{H}, and \cite{OR}, we use the concept of {\em perfect Bayesian equilibrium} (PBE) as a solution concept. We show the existence of a semi-separating PBE, in which the strong type always chooses the high salary, whereas the weak type randomizes between the low and the high salary. 

In Section~\ref{sec2}, a precise formulation of our hiring-and-firing game is presented, and Section~\ref{sec3} 
recalls some standard results from filtering theory. In Section~\ref{sec4} we argue heuristically to derive a candidate equilibrium of strategies; the candidate equilibrium is then verified to in fact be a PBE.
While our Bayesian game set-up is rather simplistic with only two possible types of the employee and two possible salary claims, it may serve as a benchmark for more involved problems. This is illustrated in Section~\ref{sec5}, where a few such extensions are briefly discussed. For these extensions, the solution of the benchmark case presented in Sections~\ref{sec2}-\ref{sec4} is utilized.

\section{Set-up}\label{sec2}

To describe the game in further detail, let $W$ be a standard Brownian motion, and let $\mu$ be a modified Bernoulli distributed random variable independent of $W$ with
\[\P(\mu=\mu_1)=p=1-\P(\mu=\mu_0),\] 
where $\mu_0$, $\mu_1$ and $p$ are known constants with $\mu_0<\mu_1$ and $p\in(0,1)$.
We assume that the employee (Player 1) generates to the employer (Player 2) a payment stream modeled as the increments of a process 
\[X_t = \mu t + \sigma W_t,\]
where $\sigma$ is a positive constant. The random variable $\mu$ will be referred to as {\em the
capacity} of the employee.

Player~1 knows their own type, i.e. the realization of their capacity $\mu$, and gives at the initial time $t=0$ a salary claim $C$ in the set $\{c_0,c_1\}$, 
where $0<c_0<c_1$. More precisely, allowing for randomized strategies, a strategy of Player~1 consists of a pair 
$a=(a_0,a_1)\in\mathcal P$, where $\mathcal P=[0,1]^2$ is the unit square. Here $a_i$ represents the conditional probability 
of choosing $C=c_1$ given that Player~1 is of type $i$, $i=0,1$. 

To describe the possible strategies of Player~2, denote by $\mathbb F^{X}=(\mathcal F^{X}_t)_{t\geq 0}$ 
the augmentation of the filtration generated by the process $X$, and by $\mathbb F=(\mathcal F_t^{X,C})_{t\geq 0}$ the augmentation of the filtration generated by the process $X$ and the random variable $C$.
Also, let $\mathcal T^X$ be the collection of $\mathbb F^{X}$-stopping times, and  
$\mathcal T$ be the collection of $\mathbb F$-stopping times.
Clearly, since $C$ only takes two possible values, any stopping time $\tau\in\mathcal T$ 
can be decomposed as
\begin{equation}
\label{tau}
\tau=\left\{\begin{array}{ll}
\tau_0 & \mbox{on }\{C=c_0\}\\
\tau_1 & \mbox{on }\{C=c_1\},\end{array}\right.
\end{equation}
where $(\tau_0,\tau_1)\in\mathcal T^X\times\mathcal T^X$. Conversely, defining $\tau$ by \eqref{tau} for a given 
pair $(\tau_0,\tau_1)\in\mathcal T^X\times\mathcal T^X$ yields that $\tau\in\mathcal T$.
Thus we may identify $\mathcal T$ with $\mathcal T^X\times\mathcal T^X$, and we therefore write  
$\tau=(\tau_0,\tau_1)$.

In addition to a pair $(a,\tau)\in\mathcal P\times \mathcal T$ of strategies, the definition of a perfect Bayesian equilibrium (given below) also requires the specification of a belief system $\Pi_0=(\Pi_0^0,\Pi_0^1)\in\mathcal P$. Here $\Pi_0^i$ represents
the probability that Player~2 assigns to the event $\{\mu=\mu_1\}$ conditional on the signal $C=c_i$, $i=0,1$.

The payoff structure of the game is now described as follows. Up to the stopping time $\tau$, Player 1 receives compensation for their work at rate $C$ per unit of time. Player 2, on the other hand, receives increments of the net payment stream 
\begin{equation}
\label{net}
X_t-Ct=(\mu-C)t + \sigma W_t
\end{equation}
per unit of time. Both players seek to maximize their expected discounted future payoff.
More precisely, for a given triple $(a,\tau,\Pi_0)\in \mathcal P\times\mathcal T\times \mathcal P$ with $a=(a_0,a_1)$, $\tau=(\tau_0,\tau_1)$
and $\Pi_0=(\Pi_0^0,\Pi_0^1)$, and for a given discount rate $r>0$, define 
\[J_1^0(a,\tau)=(1-a_0)\E\left[\int_0^{\tau_0} e^{-rt}c_0\,dt\Big\vert \mu=\mu_0\right] + a_0\E\left[\int_0^{\tau_1} e^{-rt}c_1\,dt\Big\vert \mu=\mu_0\right]\]
and
\[J_1^1(a,\tau)=(1-a_1)\E\left[\int_0^{\tau_0} e^{-rt}c_0\,dt\Big\vert \mu=\mu_1\right] + a_1\E\left[\int_0^{\tau_1} e^{-rt}c_1\,dt\Big\vert \mu=\mu_1\right].\]
Then $J_1^i$ represents the expected payoff for Player~1 given the capacity $\mu=\mu_i$, $i=0,1$.
Similarly,  define
\[J_2^0(\tau,\Pi_0)=\E_{\Pi_0^0}\left[\int_0^{\tau_0} e^{-rt}(\mu-c_0)\,dt\right] \]
and 
\[J_2^1(\tau,\Pi_0)=\E_{\Pi_0^1}\left[\int_0^{\tau_1} e^{-rt}(\mu-c_1)\,dt\right] ,\]
so that $J_2^i$ is the expected payoff for Player~2 given that $C=c_i$, $i=0,1$. Here the sub-index in 
the expected value indicates that the expected value is calculated using the belief system $\Pi_0$ as initial probability
of the type $\mu=\mu_1$.

\begin{remark}
Note that Player~1 is equipped with randomised strategies, whereas Player~2 is not. 
The intuitive reason for this is as follows.
Player~1 acts at time 0 (revealing the realization of $C$), and once this is done, the game collapses to a single-player optimization problem
of choosing a non-anticipative stopping strategy 
for Player~2. In such a Markovian optimal stopping problem (cf. Subsection~~\ref{easyemployer}), however,
the optimal value is attained for a hitting time 
(a pure strategy), and there is typically no other
strategies (e.g. mixed strategies) that are optimal.
\end{remark}

\begin{remark}
While a key ingredient in our set-up is asymmetric information about the capacity $\mu$, we point out that the set-up itself, including the numerical values of all parameters $p$, $\mu_0$, $\mu_1$, $\sigma$, $r$, $c_0$ and $c_1$, is common knowledge to both players.
\end{remark}

We now introduce the solution concept that we use below.

\begin{definition} 
{\bf (Perfect Bayesian equilibrium.)}
We call a triplet $(a^*,\tau^*,\Pi_0)\in\mathcal P\times\mathcal T \times \mathcal P$ a perfect Bayesian equilibrium (PBE) if the following conditions are satisfied. 
\begin{itemize}
\item[\normalfont{(A)
Rationality:}]
\[J^i_1(a^*,\tau^*)\geq J^i_1(a,\tau^*)\]
for $i=0,1$, and
\[J^i_2(\tau^*,\Pi_0)\geq J^i_2(\tau,\Pi_0),\]
$i=0,1$,
for all pairs $(a,\tau)\in\mathcal P\times\mathcal T$.
\item[\normalfont{(B) Bayesian updating:}]
If $\min \{a^*_0,a^*_1\}<1$, then 
\[\Pi_0^0=\frac{p(1-a^*_1)}{p(1-a^*_1) + (1-p)(1-a^*_0)},\]
and if $\max\{a^*_0,a^*_1\}>0$, then 
\[\Pi_0^1=\frac{pa^*_1}{pa^*_1+ (1-p)a^*_0}.\]
\end{itemize}
\end{definition}

\begin{remark}
In the definition above, the first condition (A) requires the strategy triplet to form an equilibrium in the usual sense that neither of the players wants to unilaterally deviate from it.
The second condition (B) requires Player~2 to update their belief system using Bayes' rule on events with positive probability. If $a_0^*=a_1^*=1$ then Player 1 always chooses $C=c_1$, and if $a_0^*=a_1^*=0$, then Player 1 always chooses $C=c_0$, and in these cases the belief system can be chosen with no restriction.
\end{remark}

\begin{remark}
If $a_0^*=a_1^*\in\{0,1\}$, then the PBE is of {\em pooling type};
if $a_0^*=1-a_1^*\in\{0,1\}$, then the PBE is {\em separating}; finally, if $a^*_i\in\{0,1\}$ and $a_{1-i}^*\notin\{0,1\}$, then the PBE is {\em semi-separating}.
\end{remark}

\section{Filtering}
\label{sec3}

From the perspective of Player~2, the problem is a two-source learning problem: at time $t=0$, 
the salary claim $C$ is observed and the prior distribution of $\mu$ is updated in accordance with the 
specified belief system; at subsequent times $t>0$, the posterior distribution is updated 
using observations of $X$.

Given $\pi\in[0,1]$, define the process $\tilde\Pi:=\tilde\Pi^{\pi}$ by
\[\tilde \Pi_t:=\P_\pi(\mu=\mu_1\vert \F_t^{X}),\]
where the index $\pi$ indicates that the conditional probability is calculated using an initial estimate $\pi$ for
the event $\{\mu=\mu_1\}$. Thus $\tilde\Pi$ is the probability that $\mu=\mu_1$
conditioned merely on observations of $X$, and calculated with an initial belief $\pi$. 
It is well-known from filtering theory (see \cite{LS})
that the conditional probability $\tilde\Pi$ satisfies
\[d\tilde\Pi_t=\omega\tilde\Pi_t(1-\tilde\Pi_t)\,d\hat W_t,\]
where $\omega:=(\mu_1-\mu_0)/\sigma$ is the {\em signal-to-noise ratio}
and the {\em innovations process}
\[\hat W_t:=\frac{1}{\sigma}\left(X_t-\int_0^t (\mu_0+(\mu_1-\mu_0)\tilde\Pi_s)\,ds\right)\]
is an $\F^X-$Brownian motion. 
In particular, the process $\tilde \Pi$ is strong Markov.

Now, given a belief system $\Pi_0=(\Pi_0^0,\Pi_0^1)\in\mathcal P$, 
we define the conditional probability process
\begin{equation}
    \label{Pi}
\Pi_t:=\left\{\begin{array}{ll}
\tilde\Pi_t^{\Pi_0^0} & \mbox{on }\{C=c_0\}\\
\tilde\Pi_t^{\Pi_0^1} & \mbox{on }\{C=c_1\}.\end{array}\right.
\end{equation}
If the Bayesian updating property (B) holds, then $\Pi_t$ coincides with 
\[\P(\mu=\mu_1\vert \F^{X,C}_t),\]
on the event $\{C=c_i\}$ provided that $\P(C=c_i)>0$.

\begin{lemma}\label{lem}
Let $(\tau,\Pi_0)\in\mathcal T\times\mathcal P$. Then, for $i=0,1$, we have 
\[J_2^i(\tau,\Pi_0)=\E_{\Pi_0^i}\left[\int_0^{\tau_i}e^{-rt}(\mu_0-c_i+(\mu_1-\mu_0)\Pi_t) dt\right].\]
\end{lemma}

\begin{proof}
By conditioning, 
\[\E_{\Pi_0^i}\left[\int_0^{\tau_i} e^{-rt}\mu\,dt\right] =\E_{\Pi_0^i}\left[\mu\frac{1-e^{-r\tau_i}}{r}\right]
=\E_{\Pi_0^i}\left[(\mu_0 +(\mu_1-\mu_0)\Pi_{\tau_i})\frac{1-e^{-r\tau_i}}{r}\right],\]
where $\Pi_t:=\P_{\Pi_0^i}\left(\mu=\mu_1\vert \F^X_t\right)$.
Moreover, by an application of Ito's formula and optional sampling, 
\[\E_{\Pi_0^i}\left[(\mu_0 +(\mu_1-\mu_0)\Pi_{\tau_i})\frac{1-e^{-r\tau_i}}{r}\right]
=\E_{\Pi_0^i}\left[\int_0^{\tau_i} e^{-rt} (\mu_0 +(\mu_1-\mu_0)\Pi_t)\,dt\right].\]
Consequently, 
\begin{eqnarray*}
J_2^i(\tau,\Pi_0) &=& \E_{\Pi_0^i}\left[\int_0^{\tau_i} e^{-rt}(\mu-c_i)\,dt\right] \\
&=& \E_{\Pi_0^i}\left[\int_0^{\tau_i} e^{-rt} (\mu_0 -c_i+(\mu_1-\mu_0)\Pi_t)\,dt\right].
\end{eqnarray*}
\end{proof}

\section{A semi-separating PBE}
\label{sec4}

Note that if $c_0\geq \mu_1$, then the net drift $\mu-C$ in \eqref{net} is non-positive, and 
Player~2 should choose immediate firing ($\tau=0$).
Similarly, if $\mu_0\geq c_1$, then $\tau=\infty$ would always be optimal. 
Thus, to rule out degenerate cases, a minimal assumption is that $\mu_0< c_1<\mu_1$. 
Moreover, we will make the additional assumption that $c_0\leq\mu_0$ so that the net drift
$\mu-C$  in \eqref{net} is non-negative on the event $\{C=c_0\}$.
We thus impose the parameter ordering
\begin{equation}
\label{ass}
0<c_0\leq \mu_0<c_1<\mu_1.
\end{equation}

It is straightforward to see that there is no PBE of separating type. Indeed, in a separating equilibrium with $a^*=(0,1)$, the strong type would never be fired, and therefore the weak type would have an incentive to deviate and choose $c_1$; similarly, if $a^*=(1,0)$, then the weak type would be fired immediately, and thus again have an incentive to deviate.
The aim of the current section is to derive a perfect Bayesian equilibrium of semi-separating type under the assumption \eqref{ass}.
In Subsections~\ref{easyemployer}-\ref{easyemployee} we use intuitive arguments to derive a candidate equilibrium,
which is then verified in Subsection~\ref{easyverification}.

\subsection{The employer's perspective} \label{easyemployer}

Under the assumption \eqref{ass}, the lower salary level $c_0$ is smaller than the capacity $\mu$ with probability one.
Thus, if \eqref{ass} holds, then it is clear that if Player~1 chooses $C=c_0$, then an optimal 
response for Player~2 should be to choose $\tau_0=\infty$.

On the other hand, on the event $\{C=c_1\}$, Player~2 would stop if there is sufficient evidence that $\mu=\mu_0$. More precisely, we expect a boundary level $b$ such that
\begin{equation}
\label{tau1}
\tau_1:=\inf\{t\geq 0:\Pi_t\leq b\}
\end{equation}
is an optimal response for Player~2. To determine $b$, standard optimal stopping theory
based on the dynamic programming principle (see, e.g., \cite{Sh}) suggests that the pair 
$(V,b)$, where 
\begin{eqnarray*}
V(\pi) &:=&
\sup_{\tau}\E\left[\int_0^\tau e^{-rt}(\mu_0-c_1 +(\mu_1-\mu_0)\tilde\Pi^\pi_t)\,dt\right],
\end{eqnarray*}
solves the free-boundary problem
\begin{equation}
\label{fbp}
\left\{\begin{array}{rl}
\mathcal L V +\mu_0-c_1+(\mu_1-\mu_0)\pi=0 & \pi\in(b,1)\\
V(b)=0\\
V_\pi(b)=0\\
V(1-)= (\mu_1-c_1)/r,\end{array}\right.
\end{equation}
where 
\[\mathcal L =\frac{1}{2}\omega^2\pi^2(1-\pi)^2\frac{d^2}{d\pi^2}-r.\]
Here the two boundary conditions at $b$ constitute the so-called condition of smooth fit, and the boundary condition at $\pi=1$ corresponds to receiving (discounted) payments at rate $\mu_1-c_1$ until time $\tau=\infty$. Note that the ODE in \eqref{fbp}
is of second order, so there are two degrees of freedom; additionally, the boundary $b$ is unknown. On the other hand, there are three boundary conditions, so one would expect that \eqref{fbp} is well-posed.

To solve the free-boundary problem \eqref{fbp}, one readily verifies that the general solution of the ODE is given by 
\[V(\pi)=A_1(1-\pi)\left(\frac{\pi}{1-\pi}\right)^{\gamma_1} + A_2(1-\pi)\left(\frac{\pi}{1-\pi}\right)^{\gamma_2} + \frac{\mu_0-c_1+(\mu_1-\mu_0)\pi}{r},\]
where $\gamma_1<0$ and $\gamma_2>1$ are the solutions of the quadratic equation
\begin{equation}
\label{quadratic}
\gamma^2 -\gamma -\frac{2r}{\omega^2} =0,
\end{equation}
and $A_1$ and $A_2$ are arbitrary constants.
Imposing the boundary condition at $\pi=1$, we must have $A_2=0$, 
and so the two remaining boundary conditions yield
\[\left\{\begin{array}{ll}
A_1(1-b)\left(\frac{b}{1-b}\right)^{\gamma_1} +\frac{\mu_0-c_1+(\mu_1-\mu_0)b}{r}=0\\
A_1(\gamma_1-b)\left(\frac{b}{1-b}\right)^{\gamma_1}+\frac{(\mu_1-\mu_0)b}{r}=0.
\end{array}\right.\]
Eliminating $A_1$, we find that 
\begin{equation}
\label{b}
b=\frac{-(c_1-\mu_0)\gamma_1}{\mu_1-c_1-(\mu_1-\mu_0)\gamma_1}.
\end{equation}

Thus the candidate optimal response for Player~2 when Player~1 chooses the higher salary $C=c_1$ is to stop when the conditional probability process $\Pi$ falls below the constant boundary $b$ in \eqref{b}.
Moreover, the candidate value for the employer is then
\[
    V(\pi)=\left\{\begin{array}{cl}
\frac{c_1-\mu_0-(\mu_1-\mu_0)b}{r}\left(\frac{\pi(1-b)}{b(1-\pi)}\right)^{\gamma_1}\frac{1-\pi}{1-b} + \frac{\mu_0-c_1+(\mu_1-\mu_0)\pi}{r} & \pi>b\\
0 & \pi\leq b.\end{array}\right.\]
For a graphical illustration of the function $V$ and the threshold $b$, see Figure~\ref{fig:V_plot}.

\begin{figure}[h]
    \centering
    \captionsetup{width=.8\textwidth}
    \includegraphics[width=0.9\textwidth]{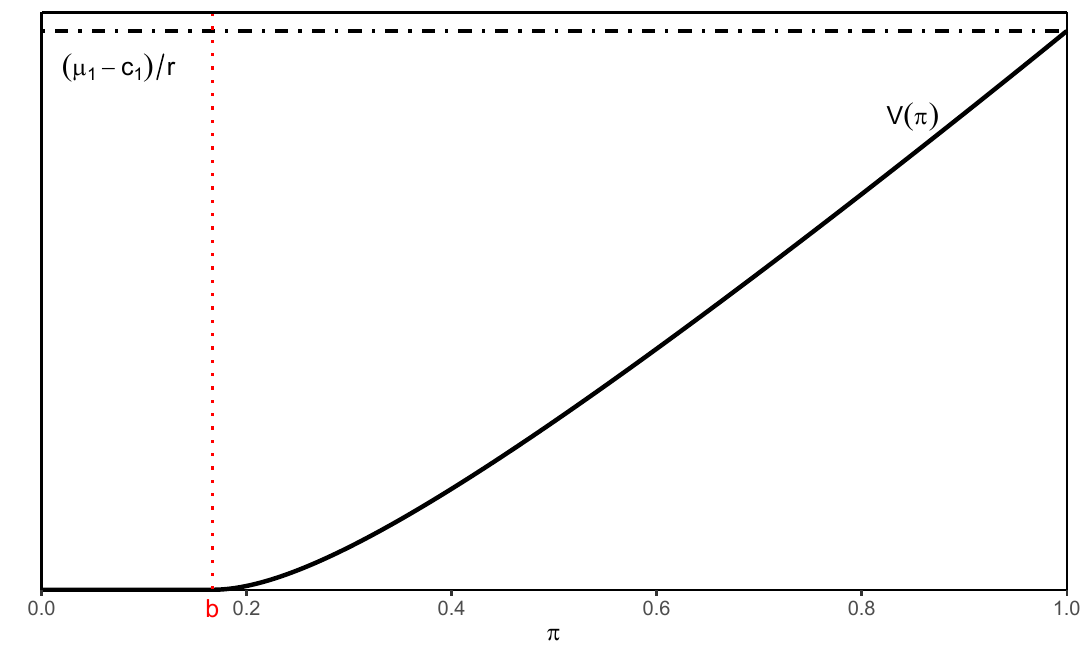}
    \caption{The value function $V(\pi)$ of the employer in the case when $C=c_1$ is chosen. The parameter values chosen for this example figure are $c_1=1.5$, $\mu_0=1.4$, $\mu_1=1.7$, $r=0.05$ and $\sigma=1$.
The value function attains positive values only after the boundary level $b \approx 0.167$, and it approaches its maximum value $(\mu_1-c_1)/r$ for $\pi$ close to 1.}
    \label{fig:V_plot}
\end{figure}

\subsection{The employee's perspective}\label{easyemployee}

We now take the perspective of Player 1. We will construct an equilibrium in which Player~1 
always chooses $C=c_1$ on the event $\{\mu=\mu_1\}$, and on the event $\{\mu=\mu_0\}$
uses a strategy such that $\P(C=c_1\vert \mu=\mu_0)=a_0=1-\P(C=c_0\vert \mu=\mu_0)$ for some $a_0\in[0,1]$ to be determined. Thus, in the notation of Section~\ref{sec2}, we consider the strategy $a=(a_0,1)\in\mathcal P$.

As noted above, on the event $\{C=c_0\}$, Player~2 would use $\tau_0=\infty$.
By the {\em indifference principle} in game theory (see, e.g., \cite{F} or \cite{L}), to have an equilibrium with a strategy pair
$(a^*,\tau^*)$ in which Player~1 uses a mixed strategy $a^*=(a_0,1)$
with $a_0\in(0,1)$ and Player~2 uses $\tau^*=(\infty,\tau_1)$ with $\tau_1$ as in \eqref{tau1}, we need that the expected payoffs 
$J^0_1((0,1),\tau^*)$ and $J^0_1((1,1),\tau^*)$ coincide.
Clearly, choosing $C=c_0$ gives the expected payoff 
\[J^0_1((0,1),(\infty,\tau_1))=c_0/r\]
for Player~1.

To determine the expected payoff $J^0_1((1,1),\tau^*)$ for Player~1 when $C=c_1$ is chosen,
note that on the event $\{C=c_1\}$, Player~2 would first re-evaluate the 
probability that Player~1 has the larger capacity $\mu=\mu_1$ according to the specified belief system $\Pi_0$. Moreover, by the Bayesian updating requirement of the belief system, we have
\[\Pi_0^1=\P(\mu=\mu_1\vert C=c_1)=\frac{\P(\mu=\mu_1, C=c_1)}{\P( C=c_1)}=
\frac{p}{p+(1-p)a_0}.\]
Thus $\Pi_t$ makes an initial jump from $\Pi_{0-}=p$ up to $\Pi^1_0=\frac{p}{p+(1-p)a_0}\geq p$, 
and then it diffuses with dynamics 
\[d\Pi_t=\omega\Pi_t(1-\Pi_t)\,d\hat W_t.\]
From the perspective of Player~1, however, $\hat W$ is not a Brownian motion since Player~1 knows the true value of $\mu$. Instead, 
\begin{eqnarray*}
d\Pi_t=\omega\Pi_t(1-\Pi_t)\,d\hat W_t = -\omega^2\Pi_t^2(1-\Pi_t)\,dt + \omega\Pi_t(1-\Pi_t)\,dW_t.
\end{eqnarray*}
Consequently, if Player~1 chooses $C=c_1$, then their value is $U(\Pi^1_0)$, where
\begin{equation}\label{U}
U(\pi) :=\E_\pi\left[\int_0^{\tau_1}e^{-rt}c_1\,dt\Big\vert \mu=\mu_0\right]
\end{equation}
solves
\begin{equation}
\label{bp}
\left\{\begin{array}{rl}
\frac{\omega^2\pi^2(1-\pi)^2}{2}U_{\pi\pi}-\omega^2\pi^2(1-\pi)U_{\pi}-rU +c_1=0 & \pi\in(b,1)\\
U(b)=0\\
U(1-)= c_1/r.\end{array}\right.
\end{equation}
The ODE in \eqref{bp} has general solution 
\[U(\pi)=B_1\left(\frac{\pi}{1-\pi}\right)^{\gamma_1}  + B_2\left(\frac{\pi}{1-\pi}\right)^{\gamma_2}
+c_1/r,\]
where $\gamma_1<0$ and $\gamma_2>1$ are the solutions of \eqref{quadratic} as before,
and $B_1$ and $B_2$ are arbitrary constants. Similarly as above, the boundary condition at $\pi=1$ yields $B_2=0$, and then the boundary condition at $\pi=b$ gives
\[B_1=\frac{-c_1}{r}\left(\frac{b}{1-b}\right)^{-\gamma_1},\]
so 
\[U(\pi)=\left\{\begin{array}{cl}
\frac{c_1}{r}\left(1-\left(\frac{\pi(1-b)}{(1-\pi)b}\right)^{\gamma_1}\right) & \pi>b\\
0 &\pi\leq b.\end{array}\right.\]

Now recall that by the indifference principle we are looking for $a_0\in(0,1)$ so that 
\[\frac{c_0}{r}= U\left(\Pi_0^1\right)
=U\left(\frac{p}{p+(1-p)a_0}\right).\]
This is possible only if $U(p)<c_0/r$, i.e.\ if
\[ \frac{p}{1-p}< \frac{b}{1-b}\left(1-\frac{c_0}{c_1}\right)^{1/\gamma_1}.\]
Equivalently, we need to have 
\begin{equation}
\label{phat}
p<\hat p:=\frac{b(1-\frac{c_0}{c_1})^{1/\gamma_1}}{1-b+b(1-\frac{c_0}{c_1})^{1/\gamma_1}},
\end{equation}
where $\hat p$ is the indifference point so that $U(\hat p)=c_0/r$.
Moreover, in that case, $a_0$ should be chosen so that
\[\frac{p}{p+(1-p)a_0}=\frac{b(1-\frac{c_0}{c_1})^{1/\gamma_1}}{1-b+b(1-\frac{c_0}{c_1})^{1/\gamma_1}},\]
i.e.
\[a_0=\frac{p(1-b)}{(1-p)b(1-\frac{c_0}{c_1})^{1/\gamma_1}}.\]

That is, the candidate optimal strategy for Player~1 is as follows. If $\mu=\mu_1$, then the high salary is chosen with probability one ($a_1=1$). On the other hand, if Player~1 is of the weak type ($\mu=\mu_0$), then 
the high salary $c_1$ is chosen with probability $a_0$, where 
\[a_0=\left\{\begin{array}{cl}
\frac{p(1-b)}{(1-p)b(1-\frac{c_0}{c_1})^{1/\gamma_1}} & p<\hat p\\
1 & p\geq \hat p.\end{array}\right.\]
For a graphical illustration of the value function $U$ and the indifference point $\hat p$, see Figure~\ref{fig:U_plot}.
\begin{figure}[h]
    \centering
    \captionsetup{width=.8\textwidth}
    \includegraphics[width=0.9\textwidth]{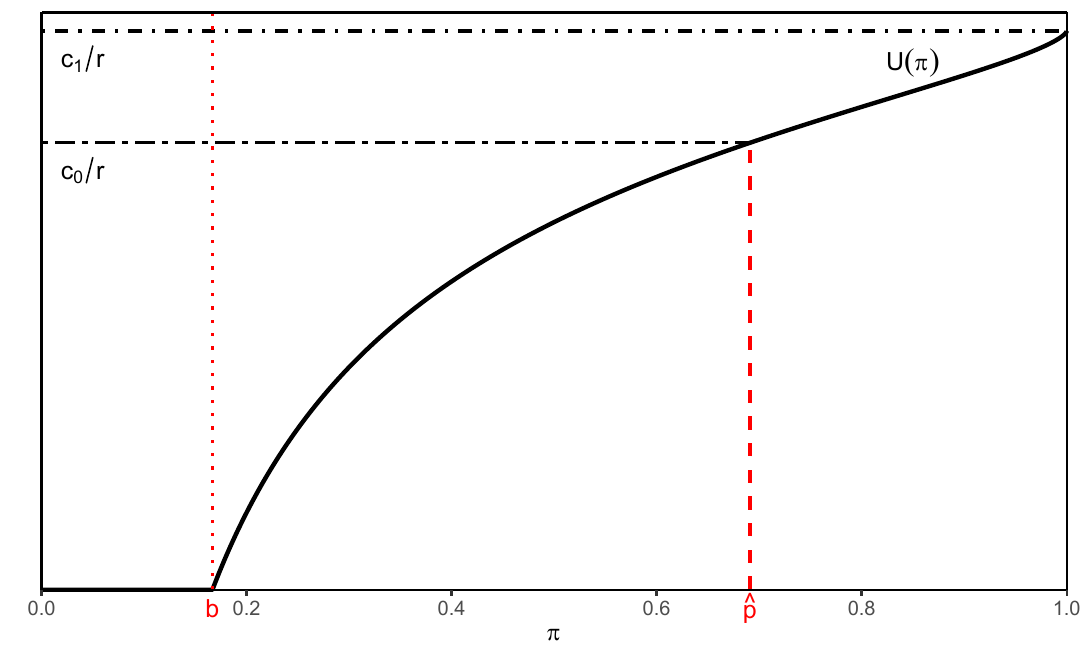}
    \caption{The value function $U(\pi)$ for the weak type employee when the high salary $C = c_1$ is chosen. The parameter values of $c_1$, $\mu_0$, $\mu_1$, $r$ and $\sigma$ are the same as in Figure~\ref{fig:V_plot}, and $c_0=1.2$. Here $\hat p$ is the unique value
so that $U(\hat p)=c_0/r$.}
    \label{fig:U_plot}
\end{figure}

\subsection{Verification of equilibrium}\label{easyverification}
We now summarize the strategies described above; in Theorem~\ref{main} we then verify that these strategies together constitute a perfect Bayesian equilibrium.

Let 
\[b=\frac{-(c_1-\mu_0)\gamma_1}{\mu_1-c_1-(\mu_1-\mu_0)\gamma_1}\]
as in \eqref{b} above, and define the strategy $a^*=(a_0^*,a^*_1)$ of Player~1 by 
\[a^*_0=\left\{\begin{array}{cl}
\frac{p(1-b)}{(1-p)b(1-\frac{c_0}{c_1})^{1/\gamma_1}} & p<\hat p\\
1 & p\geq\hat p\end{array}\right.\]
and $a_1^*=1$, where $\hat p$ is as in \eqref{phat}.
Moreover, let $\Pi_0:=(\Pi_0^0,\Pi_0^1)=(0,\hat p\vee p)$,
and let $\tau^*=(\tau^*_0,\tau^*_1)$ be defined by
\[\tau^*_0:=\infty\]
and
\[\tau^*_1:=\inf\{t\geq 0:\tilde\Pi_t^{\Pi_0^1}\leq b\}.\]

\begin{theorem}\label{main}
Assume that \eqref{ass} holds.
Then the triplet $(a^*,\tau^*, \Pi_0)$ specified above is a perfect Bayesian equilibrium. Moreover, if $p<\hat p$, the equilibrium is 
semi-separating; if $p\geq \hat p$, then the equilibrium is of pooling type.
\end{theorem}

\begin{proof}
We first note that, by construction, the belief system $\Pi_0$ satisfies the Bayesian updating property.
The proof of rationality is divided into two parts.

\vspace{4mm}\noindent
{\bf Optimality of $\tau^*$.} 
First note that $\Pi_0^0=0$ yields that 
\[J_2^0(\tau,\Pi_0)=\E\left[\int_0^{\tau_0} e^{-rt}(\mu_0-c_0)\,dt\Big \vert \mu=\mu_0\right]\leq \frac{\mu_0-c_0}{r} = J_2^0(\tau^*,\Pi_0)\]
for any $\tau\in\mathcal T$, so $\tau_0^*=\infty$ is a rational response to $C=c_0$.

Next, if the employer observes the event $\{C=c_1\}$, then the stopping time 
\[\tau^*_1:=\inf\{t\geq 0 :Z_t\leq b\}\] 
is used, where
\[ Z_t:=\tilde{\Pi}_t^{\Pi^1_0}=\P_{\Pi_0^1}(\mu=\mu_1\vert \F_t^{X})\]
(cf. \eqref{Pi}).
By Section~\ref{sec3}, 
\[dZ_t=\omega Z_t(1-Z_t)\,d\hat W_t,\]
so an application of Ito's formula together with \eqref{fbp} shows that 
\[Y_t:=e^{-rt}V(Z_t)+ \int_0^t e^{-rs}(\mu_0-c_1+(\mu_1-\mu_0)Z_s)ds\]
is a bounded supermartingale. For any stopping time $\tau'=(\tau_0',\tau_1')\in\mathcal T$, optional sampling therefore gives that 
\begin{eqnarray*}
V(\Pi_0^1) &\geq& \E\left[e^{-r(T\wedge\tau'_1)}V(Z_{T\wedge\tau'_1})+ \int_0^{T\wedge\tau'_1} e^{-rt}(\mu_0-c_1+(\mu_1-\mu_0)Z_t)dt\right]\\
&\geq& \E\left[ \int_0^{T\wedge\tau'_1} e^{-rt}(\mu_0-c_1+(\mu_1-\mu_0)Z_t)dt\right]\\
&\to& \E\left[ \int_0^{\tau'_1} e^{-rt}(\mu_0-c_1+(\mu_1-\mu_0)Z_t)dt\right]
\end{eqnarray*}
as $T\to\infty$ by bounded convergence. Since  
\[\E\left[ \int_0^{\tau'_1} e^{-rt}(\mu_0-c_1+(\mu_1-\mu_0)Z_t)dt\right]=J_2^1(\tau',\Pi_0)\]
by Lemma~\ref{lem}, we find that 
\begin{equation}
\label{ineq}
J_2^1(\tau',\Pi_0)\leq V(\Pi_0^1)
\end{equation}
for all $\tau'\in\mathcal T$.

Furthermore, for $\tau^*$, the stopped process $Y_{t\wedge\tau^*_1}$ is a martingale, so optional sampling and bounded convergence give
\begin{eqnarray*}
V(\Pi_0^1) &=& \E\left[e^{-r(T\wedge\tau^*_1)}V(Z_{T\wedge\tau^*_1})+ \int_0^{T\wedge\tau^*_1} e^{-rt}(\mu_0-c_1+(\mu_1-\mu_0)Z_t)dt\right]\\
&\to& \E\left[ \int_0^{\tau^*_1} e^{-rt}(\mu_0-c_1+(\mu_1-\mu_0)Z_t)dt\right] = J_2^1(\tau^*,\Pi_0)
\end{eqnarray*}
as $T\to\infty$, which together with \eqref{ineq} implies that $\tau^*_1$ is an optimal response to $C=c_1$.

\vspace{4mm}
\noindent
{\bf Optimality of $a^*$.}
We have that 
\begin{eqnarray*}
J^0_1(a,\tau^*) &=& (1-a_0)\frac{c_0}{r} + a_0\E_{\Pi_0^1}\left[\int_0^{\tau_1^*} e^{-rt}c_1\,dt\Big\vert\mu=\mu_0\right]\\
&=& (1-a_0)\frac{c_0}{r} + a_0U(\hat p\vee p)\\
&\leq& (1-a^*_0)\frac{c_0}{r} + a^*_0U(\hat p\vee p) = J^0_1(a^*,\tau^*),
\end{eqnarray*}
where the inequality holds since if $p>\hat p$ then we have $U(\hat p\vee p)=U(p)\geq c_0/r$ and $a_0^*=1\geq a_0$, and if 
$p\leq \hat p$ then we have $U(\hat p\vee p)=U(\hat p)=c_0/r$.

Similarly, 
\begin{eqnarray*}
J^1_1(a,\tau^*) &=& (1-a_1)\frac{c_0}{r} + a_1\E_{\Pi_0^1}\left[\int_0^{\tau_1^*} e^{-rt}c_1\,dt\Big\vert\mu=\mu_1\right]\\
&\leq& \E_{\Pi_0^1}\left[\int_0^{\tau_1^*} e^{-rt}c_1\,dt\Big\vert\mu=\mu_1\right] = J^1_1(a^*,\tau^*),
\end{eqnarray*}
where the inequality follows from the inequalities
\[\E_{\Pi_0^1}\left[\int_0^{\tau_1^*} e^{-rt}c_1\,dt\Big\vert\mu=\mu_1\right] \geq 
\E_{\Pi_0^1}\left[\int_0^{\tau_1^*} e^{-rt}c_1\,dt\Big\vert\mu=\mu_0\right] =U(\Pi_0^1)\geq c_0/r.\]
\end{proof}

\begin{remark}
\label{rem}
As is often the case for signaling games, there is no uniqueness of PBEs. Indeed, 
consider the strategy pair $(a,\tau)$, where $a=(0,0)$ and $\tau=(\infty,0)$; in words, Player~1 always chooses $C=c_0$ (regardless of their type) and Player~2 never stops if $C=c_0$ and stops immediately if $C=c_1$.
Then also $(a,\tau,\Pi_0)$ with $\Pi_0=(p,\Pi_0^1)$ is a perfect Bayesian equilibrium (of pooling type) provided the belief $\Pi_0^1$ is chosen small 
enough (e.g., $\Pi_0^1\leq b$). 

There have been substantial efforts in the literature to refine the notion of PBE (cf. \cite{BS} and \cite{KW}) in order to rule out some non-intuitive equilibria. Rather than taking that path, however, we simply note that in the pooling equilibrium 
both types have the same equilibrium value $c_0/r$, which is dominated by the corresponding equilibrium values in Theorem~\ref{main}, so the semi-separating PBE is preferred by the first-mover of our game.
\end{remark}

\begin{remark}
We have analyzed the game under the assumption \eqref{ass} that $c_0\leq\mu_0<c_1<\mu_1$. 
In the alternative ordering $\mu_0<c_0<c_1<\mu_1$,
the smaller salary $c_0$ provides a negative running reward for the employer in case of a weak type employee, and a semi-separating 
(or separating) equilibrium is not feasible. 
As in Remark~\ref{rem}, one can construct a pooling PBE with $a=(0,0)$, but one also obtains a 
pooling equilibrium with $a=(1,1)$, supported by a 
sufficiently small belief $\Pi_0^0$. While semi-separating and separating PBEs are not feasible, it remains an open question whether mixing between the two pooling equilibrium is possible in this parameter regime.
\end{remark}

\section{Extensions}\label{sec5}

In this section we briefly discuss a few extensions of the basic set-up presented above. All of these extensions can be easily solved using the methods of the current article, thus demonstrating the robustness of the benchmark case studied above. For the sake of brevity, we merely outline the solutions and leave out the full arguments.

\subsection{Firing cost}
In this section we consider the specification (i)-(iv) in the introduction, but with the addition that
\begin{itemize}
    \item [(v)]
    at the firing time $\tau$, Player~2 pays a firing cost $\epsilon$, where $\epsilon\in(0,\frac{c_1-\mu_0}{r})$.
\end{itemize}
Note that the assumption $\epsilon<\frac{c_1-\mu_0}{r}$ implies that the firing cost is smaller than the maximal possible loss for Player~2.
Adding assumption (v), the expected payoff for Player~2 is now
\begin{eqnarray*}
J_{2}^{\epsilon, i}(\tau,\Pi_0)&:=&\E_{\Pi_0^{i}}\left[\int_0^{\tau_i} e^{-rt}(\mu-c_i)\,dt-\epsilon e^{-r\tau_i}\mathbf{1}_{\{0<\tau_i<\infty\}}\right]\\
&=& \E_{\Pi_0^{i}} \left[\int_0^{\tau_i} e^{-rt}(\mu_0-c_i+ (\mu_1-\mu_0)\Pi_t)\,dt-\epsilon e^{-r\tau_i}\mathbf{1}_{\{0<\tau_i<\infty\}}\right]
\end{eqnarray*}
by Lemma~\ref{lem}. Note that the choice $\tau_i=0$ gives rise to no firing cost, with the interpretation that no hiring takes place.

Replacing the first boundary condition in the free-boundary problem \eqref{fbp} with $V(b)=-\epsilon$, one obtains a stopping boundary 
\[b:=b^\epsilon :=\frac{-(c_1-\mu_0-\epsilon r)\gamma_1}{\mu_1-c_1+r\epsilon-(\mu_1-\mu_0)\gamma_1},\]
where $\gamma_1$ is the negative solution to the quadratic equation \eqref{quadratic}. Due to the parameter ordering in \eqref{ass}, one verifies that indeed $b^\epsilon\in(0,1)$ when $\epsilon\in(0,\frac{c_1-\mu_0}{r})$.
Arguing as in Subsection~\ref{easyemployee},
we find that the indifference point $\hat p$ 
at which $U(\hat p)=c_0/r$ is given by
\[\hat p:=\hat{p}^\epsilon:=\frac{b^\epsilon(1-\frac{c_0}{c_1})^{1/\gamma_1}}{1-b^\epsilon+b^\epsilon(1-\frac{c_0}{c_1})^{1/\gamma_1}},\]
which leads to the candidate strategy $a^*=(a^*_0,1)$ with
\[a^*_0=\left\{\begin{array}{cl}
\frac{p(1-b^\epsilon)}{(1-p)b^\epsilon(1-\frac{c_0}{c_1})^{1/\gamma_1}} & p<\hat p\\
1 & p\geq\hat p.\end{array}\right.\]
Thus we specify a triplet $(a^{*},\tau^{*}, \Pi_0)\in\mathcal P\times\mathcal T\times\mathcal P$ by $a^{*}=(a_0^{*},1)$, where $\Pi_0=(0,\hat p\vee p)$ and $\tau^{*}=(\infty, \tau_1)$, with 
\[\tau^{*}_1=\inf\{t\geq 0:\tilde\Pi_t^{\Pi_0^1}\leq b^\epsilon\}.\] 
It is then straightforward to verify that if $V(\hat p\vee p)> 0$, then $(a^{*},\tau^{*}, \Pi_0)$ is a PBE. 
(On the other hand, if $V(\hat p\vee p)\leq 0$, then 
$((0,0),(0,0),(p,p))$, corresponding to always choosing $C=c_0$ and no hiring is an equilibrium.)

The addition of a firing cost $\epsilon\in(0,\frac{c_1-\mu_0}{r})$, together with $V(\hat p\vee p)> 0$, then constitutes a PBE with a lower stopping boundary $b^\epsilon$ for Player~2 and a higher randomizing probability $a_0^*$ for a low-type Player~1  compared to the corresponding values in the benchmark case of Section~\ref{sec4}.

\subsection{Uncertainty about one's type}
In this section we consider the same set-up as specified in (i)-(iv), but where (ii) is replaced by
\begin{itemize}
    \item [(ii')]
    At time $t=0$, Player 1 first receives a {\bf noisy observation} of their capacity $\mu$, and then presents to Player 2 a salary claim $C\in\{c_0,c_1\}$.
\end{itemize}
In this way, also Player~1 has incomplete information about their own capacity (cf. \cite{W}). 
More precisely, assume that 
the noisy signal is either of two events: 'strong belief' and 'weak belief', where
\[p_1:=\P(\mbox{strong belief}),\]
\[q:=\P(\mu=\mu_1\vert\mbox{strong belief})\]
and
\begin{equation}
\label{overest}
\P(\mu=\mu_1\vert\mbox{weak belief})=0.
\end{equation}
With this notation, the probability (denoted $p$ in the previous sections) that Player~1 has the large capacity is $p=p_1q$. 
Note that when $q=1$, this extension collapses to our original benchmark model.
Also note that a consequence of
\eqref{overest} is that Player~1 has a tendency of overestimating their capacity: if Player~1 has the weak belief, then they are always of the weak type, whereas if they have the strong belief, then they may still be of the weak type.
We adapt the strategy $a=(a_0,a_1)$ so that 
$a_0$ and $a_1$ now denote the conditional probabilities that Player~1 chooses $C=c_1$ given 'weak belief' and 'strong belief', respectively.

For simplicity, assume that $q>\hat p$, where $\hat p$ is defined in \eqref{phat}.
Now specify a triplet $(a^*,\tau^*, \Pi_0)\in\mathcal P\times\mathcal T\times\mathcal P$ by
\[a^*=(a_0^*,1),\]
\[\Pi_0=(0,\hat p\vee (p_1q))\]
and
\[\tau^*=(\infty,\tau_1^*),\]
where 
\[a_0^*=\left\{\begin{array}{cl}
\frac{p_1(q-\hat p)}{\hat p(1-p_1)} & p_1q<\hat p\\
1 & p_1q\geq\hat p\end{array}\right.
\]
and 
\[\tau^*_1=\inf\{t\geq 0:\tilde\Pi_t^{\Pi_0^1}\leq b\},\]
with $b$ as in \eqref{b}. It is then straightforward to check that $(a^*,\tau^*, \Pi_0)$ constitutes a PBE. 

\subsection{Adding an interview phase}

As a last extension of the set-up (i)-(iv), consider a situation where (iii) is replaced with
\begin{itemize}
\item[(iii')]
At time 0, Player 2 observes the salary claim
{\bf together with the result 'weak result' or 'strong result' of a test}, and subsequently also noisy observations of $\mu$, based upon which 
a choice is made of a stopping time $\tau$ to terminate the employment.
\end{itemize}

Thus, further to the salary claim $C$, Player~2 also receives information from an additional test result (cf. \cite{AP} and \cite{DG2})
with two possible outcomes. For definiteness, we assume that 
\[\P(\mbox{strong result}\vert \mu=\mu_0)=q\in(\frac{c_0}{c_1},1)\]
and 
\begin{equation}
    \label{deg}
\P(\mbox{strong result}\vert \mu=\mu_1)=1,
\end{equation}
which corresponds to a situation in which strong types always perform well in the interview, and weak types sometimes do. 
Note that the outcome of the interview test is not available for Player~1, which is similar to the situation studied in \cite{AP} and \cite{DG2}.

A belief system and a strategy of Player~2 can now be described as quadruples
\[\Pi_0=(\Pi_0^{0,weak},\Pi_0^{0,strong},\Pi_0^{1,weak},\Pi_0^{1,strong})\]
and
\[\tau=(\tau_0^{weak},\tau_0^{strong},\tau_1^{weak},\tau_1^{strong}),\] 
where for $i=0,1$, $\Pi_0^{i,weak}$, $\Pi_0^{i,weak}$, $\tau_i^{strong}$ and $\tau_i^{strong}$ are the beliefs and stopping times used provided $C=c_i$ and the test result 'weak result' or 'strong result' are observed, respectively.

By the assumption \eqref{deg}, if Player~2 observes 'weak result' in the test, then Player~1 is automatically of the weak type ($\Pi_0^{i,weak}=0$), and immediate firing ($\tau_1^{weak}=0$) would then be optimal in case $C=c_1$. 
Therefore, choosing $C=c_1$ is risky if Player 1 is of the weak type, and the obtained value from choosing $C=c_1$ (for the weak type player) is $qU(\Pi_0^{1,strong})$, with $U$ as in \eqref{U}.
We thus define the indifference point $\hat P$ via the indifference principle (cf. Section~\ref{easyemployee}) so that
\[qU(\hat P)=\frac{c_0}{r}\]
(note that $\hat P$ is well-defined since $q>c_0/c_1$ by assumption), and let 
\[a^*=(1,\frac{p(1-\hat P)}{\hat P(1-p)q}\wedge 1),\]
\[\Pi_0=(0,0,0,\Pi_0^{1,strong})\]
with 
\[\Pi_0^{1,strong}=\hat P\vee \frac{p}{p+(1-p)q},\]
and
\[\tau^*=(\infty,\infty,0, \tau_b)\]
with
\[\tau_b=\inf\{t\geq 0:\tilde\Pi_t^{\Pi_0^{1,strong}}\leq b\}.\]
It is then straightforward to check that $(a^*,\tau^*,\Pi_0)$ is a PBE. 

\begin{acknowledgement}
We are grateful to Kristoffer Glover for enlightening discussions, and for sharing with us preliminary notes on a problem in which Bayesian updating of a fund manager's skill plays a key role. We also thank Ola Andersson for stimulating discussions on the game theoretical aspects of the current work.
\end{acknowledgement}


\begin{thebibliography}{99}

\bibitem{AP}
Alós-Ferrer, C. and Prat, J.
Job market signaling and employer learning. 
{\em J. Econom. Theory} 147 (2012), no. 5, 1787-1817.


\bibitem{BS}
Banks, J. and Sobel, J.
Equilibrium selection in signaling games.
{\em Econometrica} 55 (1987), no. 3, 647-661. 

\bibitem{CR}
Cardaliaguet, P. and Rainer, C.
Stochastic differential games with asymmetric information. 
{\em Appl. Math. Optim.} 59 (2009), no. 1, 1-36.

\bibitem{CPT}
Cvitanić, J., Possamaï, D. and Touzi, N.
Dynamic programming approach to principal-agent problems. 
\emph{Finance Stoch}. 22 (2018), no. 1, 1-37.

\bibitem{DG}
Daley, B. and Green, B.
Waiting for news in the market for lemons. 
{\em Econometrica} 80 (2012), no. 4, 1433-1504. 

\bibitem{DG2}
Daley, B. and Green, B.
Market signaling with grades. 
{\em J. Econom. Theory} 151 (2014), 114-145. 

\bibitem{DGKM}
Dato, S., Grunewald, A., Kr\"akel, M. and M\"uller, D.
Asymmetric employer information, promotions, and the wage policy of firms. 
{\em Games Econom. Behav.} 100 (2016), 273-300. 

\bibitem{DA}
De Angelis, T. Optimal dividends with partial information and stopping of a degenerate reflecting diffusion. {Finance Stoch}. 24 (2020), no. 1, 71-123.

\bibitem{ELO}
 Ekstr\"om, E., Lindensj\"o, K. and Olofsson, M. 
 How to detect a salami slicer: a stochastic controller-and-stopper game with unknown competition. 
 {\em SIAM J. Control Optim.} 60 (2022), no. 1, 545-574.

\bibitem{EV}
Ekstr\"om, E. and Vaicenavicius, J. Optimal stopping of a Brownian bridge with an unknown pinning point. 
{\em Stochastic Process. Appl.} 130 (2020), no. 2, 806-823.

\bibitem{F}
Ferguson, T.
{\em A course in game theory}. World Scientific Publishing Co. Pte. Ltd., Hackensack, NJ, [2020], \copyright 2020.  ISBN: 978-981-3227-36-1; 978-981-3227-34-7; 978-981-3227-37-8. 

\bibitem{FT}
Fudenberg, D., and Tirole, J.
Perfect Bayesian Equilibrium and Sequential Equilibrium. {Journal of Economic Theory}. 53 (1991), 236-250.

\bibitem{G}
Gr\"un, C.
On Dynkin games with incomplete information. 
{\em SIAM J. Control Optim.} 51 (2013), no. 5, 4039-4065. 

\bibitem{H}
Harsanyi, J.
Games with incomplete information played by "Bayesian" players. I. The basic model. {\em Management Sci}. 14 (1967), 159-182.

\bibitem{HM}
Holmström, B., and Milgrom, P. Aggregation and linearity in the provision of intertemporal incentives. \emph{Econometrica}. 55 (1987), no. 2, 303-328. 


\bibitem{KW}
Kreps, D., and Wilson, R. Sequential equilibria. {\em Econometrica} 50 (1982), no. 4, 863-894.

\bibitem{La}
Lakner, P. Utility maximization with partial information. {\em Stochastic Process. Appl.} 56 (1995), no. 2, 247-273.

\bibitem{L}
Lindahl, L.-\AA. 
{\em An introduction to game theory – Part I.}
\copyright 2017 Lars-Åke Lindahl \& bookboon.com
ISBN 978-87-403-2133-3.


\bibitem{LS}
Liptser, R., and Shiryaev, A. {\em Statistics of random processes. I. General theory.} Second, revised and expanded edition. Applications of Mathematics (New York), 5. Stochastic Modelling and Applied Probability. Springer-Verlag, Berlin, 2001.

\bibitem{OR}
Osborne, M., and Rubinstein, A.
{\em A course in game theory}. MIT Press, Cambridge, MA, 1994.

\bibitem{Sa}
Sannikov, Y. 
A continuous-time version of the principal-agent problem. {\em Rev. Econom. Stud.} 75 (2008), no. 3, 957-984.

\bibitem{Sh}
Shiryaev, A. {\em Optimal stopping rules}. 
Applications of Mathematics, Vol. 8. Springer-Verlag, New York-Heidelberg, 1978.

\bibitem{S}
Spence, M. Job market signaling. {\em Quarterly Journal of Economics} 87 (1973), no. 3, 355-374.

\bibitem{W}
Weiss, A. (1983). A sorting-cum-learning model of education. {\em Journal of Political Economy}, 91 (1983), no. 3, 420–442. 

\end{thebibliography}
\end{document}